\documentclass[12pt,a4paper]{article}
\usepackage[utf8]{inputenc}
\usepackage{amsmath, amscd, amsthm, color}
\usepackage{amsfonts}
\usepackage{amssymb}
\usepackage{amsfonts}
\usepackage{amssymb}
\usepackage{enumerate, cite}
\usepackage{graphicx, epstopdf}
\newtheorem{theorem}{Theorem}[section]

\theoremstyle{definition}

\usepackage[left=1.00in, right=1.00in, top=1.00in, bottom=1.00in]{geometry}
\begin{document}
	\begin{center}
		\textbf{\Large{Some identities of Ramanujan's q-Continued Fraction of Order Eighteen, Twenty-Six and Thirty, and Vanishing Coefficients}}
	\end{center}
	\begin{center}
		\textbf{Raksha and B. R. Srivatsa Kumar}\\
		\vskip 0.2 cm 
		Department of Mathematics, Manipal Institute of Technology,\\
		Manipal Academy of Higher Education, 
		Manipal, India\\
		\vskip 0.2 cm 
		\text{E-mail: raksha.dpas@learner.manipal.edu, sri$_{-}$vatsabr@yahoo.com}\\
		\vskip 0.2 cm 
		
	\end{center}
	{\textbf{Abstract:}}In the present work, we established continued fractions of level eighteen, twenty six and thirty. Further, we obtained vanishing coefficients and many algebraic relations. To validate our result colored partitions are also obtained.
	\section{Introduction}
		For $q$, $\lambda$, $\mu$ $\in$ $\mathbb{C}$, the basic (or $q$-) shifted factorial $(\lambda; q)_{\mu}$ is defined by 
	\begin{equation*}
		\hspace{2cm} (\lambda ; q)_{\mu}= \prod_{j=0}^{\infty} \bigg(\frac{1-\lambda q^{j}}{1-\lambda q^{\mu + j}}\bigg) \qquad (|q|< 1 ;\lambda, \mu \in \mathbb{C} ),
	\end{equation*}
	so that
	\begin{equation*}
		(\lambda;q)_n:= \begin{cases} 
			1 & (n=0) \\
			\displaystyle{	\prod_{k=0}^{n-1}(1-\lambda q^{k})} & (n \in \mathbb{N})
		\end{cases}
	\end{equation*}
	and
	\begin{equation*}
		\hspace{4cm} (\lambda;q)_\infty=\lim_{n \rightarrow\infty}(\lambda;q)_n:= \prod_{j=0}^{\infty} (1- \lambda q^j) \qquad (|q|<1 ; \lambda \in \mathbb{C}).
	\end{equation*}
	Ramanujan’s general theta function $f(a, b)$ is defined by \cite[ p. 31, Eq. (18.1)]{Bernd}
	\begin{align}{\label{48}}
		f(a,b):=\sum_{n=-\infty}^{\infty} a^{\frac{n(n+1)}{2}}b^{\frac{n(n-1)}{2}} \qquad (|ab|<1). 
	\end{align}
	Ramanujan also rediscovered Jacobi's famous triple product identity (see \cite[ p.35, Entry 19]{Bernd})
	\begin{equation*}
		f(a,b)=(-a; ab)_\infty (-b; ab)_\infty (ab; ab)_\infty.
	\end{equation*}
	From\cite[ P. 34, Entry 8(iii)]{Bernd}, we have
	\begin{equation}
		f(-1, a)=0. \label{aaa}
	\end{equation}
	Following are the three most interesting special cases of $\eqref{48}$ \cite[ P. 36, Entry 22 (i)-(iii)]{Bernd}:
	\begin{align}
		\varphi(q) &:= f(q, q)= \sum_{n=-\infty}^{\infty}q^{n^{2}}=\frac{(-q;-q)_\infty}{(q;-q)_\infty}, \label{49}\\
		\psi(q)&:= f(q,q^3)= \sum_{n=0}^{\infty} q^{\frac{n(n+1)}{2}}=\frac{(q^2;q^2)_\infty}{(q;q^2)_\infty},\label{50}
	\end{align}
	\begin{align}
		f(-q)&:=f(-q ; -q^2)= \sum_{n=-\infty}^{\infty}(-1)^n q^{\frac{n(3n-1)}{2}}=(q;q)_\infty. \label{51}
	\end{align}
	Also, following Ramanujan’s work, we define
	\begin{equation*}
		\chi(q)=(-q;q^2)_\infty.
	\end{equation*}
	Ramanujan has recorded several continued fractions \cite{1998} and some general continued fraction identities in his notebooks. For example, Ramanujan recorded the following general continued fraction identity \cite[p. 24, Entry 12]{Bernd}.  Suppose that $a$, $b$ and $q$ are the complex numbers with $| ab|$ and $|q| < 1$ or that $ a=b^{2m+1}$ for some integer $m$. Then
\begin{equation}{\label{1}}
	\frac{(a^2q^3;q^4)_\infty (b^2q^3;q^4)_\infty}{(a^2q;q^4)_\infty (b^2q;q^4)_\infty}=\cfrac{1}{(1-ab)+\cfrac{
			(a-bq)(b-aq)}{(1-ab)(q^2+1)+ \cfrac{(a-bq^3)(b-aq^3)}{(1-ab)(q^4+1)+\ldots}}}.
\end{equation}
By specializing the values of $a$ and $b$, and taking the suitable powers    $q$, one can obtain $q$-continued fractions of particular order which satisfy
the theta functions analogous to those of $R(q)$.\\
Here we deal with  $q$-continued fraction of order eighteen, twenty six and thirty. By replacing $q$ by $q^{9/2}$ in $\eqref{1}$, setting $\left\{a=q^{1/4}, b=q^{17/4}\right\} $, $\left\{a=q^{3/4}, b=q^{15/4}\right\} $, $\left\{a=q^{5/4}, b=q^{13/4}\right\} $ and $\left\{a=q^{7/4}, b=q^{11/4}\right\} $  and simplifying ,we obtain the following four continued fractions of order eighteen respectively:
{\color{white}"}\begin{equation}{\label{4}}
	A_1(q)=\frac{f(-q^4, -q^{14})}{f(-q^5, -q^{13})}=\cfrac{(1-q^4)}{(1-q^{9/2})+\cfrac{q^{9/2}
			(1-q^{1/2})(1-q^{17/2})}{(1-q^{9/2})(q^9+1)+ \cfrac{q^{9/2}(1-q^{19/2})(1-q^{35/2})}{(1-q^{9/2})(q^{18}+1)+\ldots}}},
\end{equation}
\begin{equation}{\label{5}}
	A_2(q)=\frac{f(-q^3, -q^{15})}{f(-q^6, -q^{12})}=\cfrac{(1-q^3)}{(1-q^{9/2})+\cfrac{q^{9/2}
			(1-q^{3/2})(1-q^{15/2})}{(1-q^{9/2})(q^9+1)+ \cfrac{q^{9/2}(1-q^{21/2})(1-q^{33/2})}{(1-q^{9/2})(q^{18}+1)+\ldots}}},
\end{equation}
\begin{equation}{\label{6}}
	A_3(q)=\frac{f(-q^2, -q^{16})}{f(-q^7, -q^{11})}=\cfrac{(1-q^2)}{(1-q^{9/2})+\cfrac{q^{9/2}
			(1-q^{5/2})(1-q^{13/2})}{(1-q^{9/2})(q^9+1)+ \cfrac{q^{9/2}(1-q^{23/2})(1-q^{31/2})}{(1-q^{9/2})(q^{18}+1)+\ldots}}},
\end{equation}
and
\begin{equation}{\label{7}}
	A_4(q)=\frac{f(-q, -q^{17})}{f(-q^8, -q^{10})}=\cfrac{(1-q)}{(1-q^{9/2})+\cfrac{q^{9/2}
			(1-q^{7/2})(1-q^{11/2})}{(1-q^{9/2})(q^9+1)+ \cfrac{q^{9/2}(1-q^{25/2})(1-q^{29/2})}{(1-q^{9/2})(q^{18}+1)+\ldots}}},
\end{equation}
Similarly, we have the following continued fractions of order twenty six and thirty by choosing suitable values of $q$, $a$ and $b$ respectively:
\begin{equation}{\label{8}}
	B_1(q)=\frac{f(-q^6, -q^{20})}{f(-q^7, -q^{19})}=\cfrac{(1-q^6)}{(1-q^{13/2})+\cfrac{q^{13/2}
			(1-q^{1/2})(1-q^{25/2})}{(1-q^{13/2})(q^{13}+1)+ \cfrac{q^{13/2}(1-q^{27/2})(1-q^{51/2})}{(1-q^{13/2})(q^{26}+1)+\ldots}}},
\end{equation}

\begin{equation}{\label{9}}
	B_2(q)=\frac{f(-q^5, -q^{21})}{f(-q^8, -q^{18})}=\cfrac{(1-q^5)}{(1-q^{13/2})+\cfrac{q^{13/2}
			(1-q^{3/2})(1-q^{23/2})}{(1-q^{13/2})(q^{13}+1)+ \cfrac{q^{13/2}(1-q^{29/2})(1-q^{49/2})}{(1-q^{13/2})(q^{26}+1)+\ldots}}},
\end{equation}
\begin{equation}{\label{10}}
	B_3(q)=\frac{f(-q^4, -q^{22})}{f(-q^9, -q^{17})}=\cfrac{(1-q^4)}{(1-q^{13/2})+\cfrac{q^{13/2}
			(1-q^{5/2})(1-q^{21/2})}{(1-q^{13/2})(q^{13}+1)+ \cfrac{q^{13/2}(1-q^{31/2})(1-q^{47/2})}{(1-q^{13/2})(q^{26}+1)+\ldots}}},
\end{equation}
\begin{equation}{\label{11}}
	B_4(q)=\frac{f(-q^3, -q^{23})}{f(-q^{10}, -q^{16})}=\cfrac{(1-q^3)}{(1-q^{13/2})+\cfrac{q^{13/2}
			(1-q^{7/2})(1-q^{19/2})}{(1-q^{13/2})(q^{13}+1)+ \cfrac{q^{13/2}(1-q^{33/2})(1-q^{45/2})}{(1-q^{13/2})(q^{26}+1)+\ldots}}},
\end{equation}
\begin{equation}{\label{12}}
	B_5(q)=\frac{f(-q^2, -q^{24})}{f(-q^{11}, -q^{15})}=\cfrac{(1-q^2)}{(1-q^{13/2})+\cfrac{q^{13/2}
			(1-q^{9/2})(1-q^{17/2})}{(1-q^{13/2})(q^{13}+1)+ \cfrac{q^{13/2}(1-q^{35/2})(1-q^{43/2})}{(1-q^{13/2})(q^{26}+1)+\ldots}}},
\end{equation}
\begin{equation}{\label{13}}
	B_6(q)=\frac{f(-q, -q^{25})}{f(-q^{12}, -q^{14})}=\cfrac{(1-q)}{(1-q^{13/2})+\cfrac{q^{13/2}
			(1-q^{11/2})(1-q^{15/2})}{(1-q^{13/2})(q^{13}+1)+ \cfrac{q^{13/2}(1-q^{37/2})(1-q^{41/2})}{(1-q^{13/2})(q^{26}+1)+\ldots}}},
\end{equation}
\begin{equation}{\label{r1}}
	C_1(q)=\frac{f(-q^7, -q^{23})}{f(-q^{8}, -q^{22})}=\cfrac{(1-q^7)}{(1-q^{15/2})+\cfrac{q^{15/2}
			(1-q^{1/2})(1-q^{29/2})}{(1-q^{15/2})(q^{15}+1)+ \cfrac{q^{15/2}(1-q^{31/2})(1-q^{59/2})}{(1-q^{15/2})(q^{30}+1)+\ldots}}},
\end{equation}
\begin{equation}{\label{r2}}
	C_2(q)=\frac{f(-q^6, -q^{24})}{f(-q^{9}, -q^{21})}=\cfrac{(1-q^6)}{(1-q^{15/2})+\cfrac{q^{15/2}
			(1-q^{3/2})(1-q^{27/2})}{(1-q^{15/2})(q^{15}+1)+ \cfrac{q^{15/2}(1-q^{33/2})(1-q^{57/2})}{(1-q^{15/2})(q^{30}+1)+\ldots}}},
\end{equation} 
\begin{equation}{\label{r3}}
	C_3(q)=\frac{f(-q^5, -q^{25})}{f(-q^{10}, -q^{20})}=\cfrac{(1-q^5)}{(1-q^{15/2})+\cfrac{q^{15/2}
			(1-q^{5/2})(1-q^{25/2})}{(1-q^{15/2})(q^{15}+1)+ \cfrac{q^{15/2}(1-q^{35/2})(1-q^{55/2})}{(1-q^{15/2})(q^{30}+1)+\ldots}}},
\end{equation} 
\begin{equation}{\label{r4}}
	C_4(q)=\frac{f(-q^4, -q^{26})}{f(-q^{11}, -q^{19})}=\cfrac{(1-q^4)}{(1-q^{15/2})+\cfrac{q^{15/2}
			(1-q^{7/2})(1-q^{23/2})}{(1-q^{15/2})(q^{15}+1)+ \cfrac{q^{15/2}(1-q^{37/2})(1-q^{53/2})}{(1-q^{15/2})(q^{30}+1)+\ldots}}},
\end{equation} 
\begin{equation}{\label{r5}}
	C_5(q)=\frac{f(-q^3, -q^{27})}{f(-q^{12}, -q^{18})}=\cfrac{(1-q^3)}{(1-q^{15/2})+\cfrac{q^{15/2}
			(1-q^{9/2})(1-q^{21/2})}{(1-q^{15/2})(q^{15}+1)+ \cfrac{q^{15/2}(1-q^{39/2})(1-q^{15/2})}{(1-q^{15/2})(q^{30}+1)+\ldots}}},
\end{equation}
\begin{equation}{\label{r6}}
	C_6(q)=\frac{f(-q^2, -q^{28})}{f(-q^{13}, -q^{17})}=\cfrac{(1-q^2)}{(1-q^{15/2})+\cfrac{q^{15/2}
			(1-q^{11/2})(1-q^{19/2})}{(1-q^{15/2})(q^{15}+1)+ \cfrac{q^{15/2}(1-q^{41/2})(1-q^{49/2})}{(1-q^{15/2})(q^{30}+1)+\ldots}}},
\end{equation}{\color{white}"}
\begin{equation}{\label{r7}}
	C_7(q)=\frac{f(-q, -q^{29})}{f(-q^{14}, -q^{16})}=\cfrac{(1-q)}{(1-q^{15/2})+\cfrac{q^{15/2}
			(1-q^{13/2})(1-q^{17/2})}{(1-q^{15/2})(q^{15}+1)+ \cfrac{q^{15/2}(1-q^{43/2})(1-q^{47/2})}{(1-q^{15/2})(q^{30}+1)+\ldots}}},.
\end{equation}
In section 2, we show some results on vanishing coefficients arising from the continued fractions of order eighteen, twenty-six and thirty. In section 3, we prove some algebraic relations for the continued fractions in terms of theta functions and we deduce some partition-theoretic results using color partition for the theta function identities.  
\section{Vanishing coefficients of the continued fractions} 
In this section we offer vanishing coefficient in the series expansion of the continued  fractions.
\begin{theorem}{\label{18}} If
	\begin{eqnarray*}
		A_1^*(q)=\frac{(q^4, q^{14};q^{18})_\infty}{(q^5, q^{13}; q^{18})_\infty}= \displaystyle \sum_{n=0}^{\infty}a_n q^n, \\
		A_3^*(q)=\frac{(q^2, q^{16};q^{18})_\infty}{(q^7, q^{11}; q^{18})_\infty}= \displaystyle \sum_{n=0}^{\infty}b_n q^n, \\
		\frac{1}{A_4^*(q)}= \frac{(q^8, q^{10} ;q^{18})_\infty}{(q, q^{17}; q^{18})_\infty}= \displaystyle \sum_{n=0}^{\infty}c_n q^n,
	\end{eqnarray*}
	then, we have \\	 $\hspace{2cm}~~(i) ~~a_{9n+8}=0, ~~~(ii)~~b_{9n+8}=0, ~~~(iii)~~c_{9n+3}=0$ .
\end{theorem}
\begin{proof}
	Andrews and Bressoud \cite{1979} stated the following $p$-dissection formula
	\begin{equation}{\label{15}}
		\frac{(q^t, q^t, q^{r+s}, q^{t-r-s};q^t)_\infty}{(q^s, q^{t-s}, q^r, q^{t-r};q^t)_\infty}=\displaystyle \sum_{j=0}^{p-1} q^{jr} \frac{(q^{pt}, q^{pt}, q^{pr+s+jt},q^{(p-j)t-pr-s};q^{pt})_\infty}{(q^{jt+s},q^{(p-j)t-s},q^{pr},q^{(t-r)p};q^{pt})_\infty}
	\end{equation}
	where all the powers of $q$ in each of the infinite products on the right hand side must be multiples of $p$ and the integer $r$ must satisfy $gcd(r, p)=1$.\\
	Now, setting $t=18, s=9, r=5, p=9$ in $\eqref{15}$, we obtain
	\begin{align}{\label{45}}
		\nonumber	&\frac{(q^{18}, q^{18}, q^{14}, q^{4};q^{18})_\infty}{(q^{9}, q^{9}, q^{5}, q^{13}; q^{18})_\infty}=\frac{(q^{162}, q^{162}, q^{54}, q^{108};q^{162})_\infty}{(q^{9}, q^{153}, q^{45}, q^{117}; q^{162})_\infty}+q^5 \frac{(q^{162}, q^{162}, q^{72}, q^{90};q^{162})_\infty}{(q^{27}, q^{135}, q^{45}, q^{117}; q^{162})_\infty}\\
		\nonumber&+q^{10} \frac{(q^{162}, q^{162}, q^{90}, q^{72};q^{162})_\infty}{(q^{45}, q^{117}, q^{45}, q^{117}; q^{162})_\infty}+q^{15}\frac{(q^{162}, q^{162}, q^{108}, q^{54};q^{162})_\infty}{(q^{63}, q^{99}, q^{45}, q^{117}; q^{162})_\infty}+q^{20}\frac{(q^{162}, q^{162}, q^{126}, q^{36};q^{162})_\infty}{(q^{81}, q^{81}, q^{45}, q^{117}; q^{162})_\infty}\\
		\nonumber&+q^{25}\frac{(q^{162}, q^{162}, q^{144}, q^{18};q^{162})_\infty}{(q^{99}, q^{63}, q^{45}, q^{117}; q^{162})_\infty}+q^{30}\frac{(q^{162}, q^{162}, q^{162}, q^{0};q^{162})_\infty}{(q^{117}, q^{45}, q^{45}, q^{117}; q^{162})_\infty}+q^{35}\frac{(q^{162}, q^{162}, q^{180}, q^{-18};q^{162})_\infty}{(q^{135}, q^{27}, q^{45}, q^{117}; q^{162})_\infty}\\
		&+q^{40}\frac{(q^{162}, q^{162}, q^{198}, q^{-36};q^{162})_\infty}{(q^{153}, q^{9}, q^{45}, q^{117}; q^{162})_\infty}.
	\end{align}
	Multiplying both sides of $\eqref{45}$ by $(q^9;q^9)_\infty^2/(q^{18};q^{18})_\infty^2$ and using $\eqref{aaa}$  then simplifying, we obtain
	\begin{align}
		\nonumber	\displaystyle \sum_{n=0}^{\infty}a_n q^n=&\dfrac{(q^{9}, q^{45}, q^{117}, q^{153};q^{162})_\infty(q^{27}, q^{63}, q^{81}, q^{99}, q^{135};q^{162})_\infty^2}{(q^{54}, q^{108};q^{162})_\infty(q^{18}, q^{36}, q^{72}, q^{90}, q^{126}, q^{144};q^{162})_\infty^2}\\
		\nonumber &+q^5\dfrac{(q^{27}, q^{45}, q^{117}, q^{135};q^{162})_\infty(q^{9}, q^{63}, q^{81}, q^{99}, q^{153};q^{162})_\infty^2}{(q^{72}, q^{90};q^{162})_\infty(q^{18}, q^{36}, q^{54}, q^{108}, q^{126}, q^{144};q^{162})_\infty^2}\\
		\nonumber &+q^{10}\dfrac{(q^{9}, q^{27}, q^{63}, q^{81}, q^{99}, q^{135}, q^{153};q^{162})_\infty^2}{(q^{72}, q^{90};q^{162})_\infty(q^{18}, q^{36}, q^{54}, q^{108}, q^{126}, q^{144};q^{162})_\infty^2}
	\end{align}
	\begin{align}
		\nonumber &+q^{15}\dfrac{(q^{45},q^{63}, q^{99},  q^{117};q^{162})_\infty(q^{9}, q^{27}, q^{81}, q^{135}, q^{153};q^{162})_\infty^2}{(q^{54}, q^{108};q^{162})_\infty(q^{18}, q^{36}, q^{72}, q^{90}, q^{126}, q^{144};q^{162})_\infty^2}\\
		\nonumber &+q^{20}\dfrac{(q^{45},q^{117};q^{162})_\infty(q^{9}, q^{27}, q^{63}, q^{99}, q^{135}, q^{153};q^{162})_\infty^2}{(q^{36}, q^{126};q^{162})_\infty(q^{18}, q^{54}, q^{72}, q^{90}, q^{108}, q^{144};q^{162})_\infty^2}\\
		\nonumber &+q^{25}\dfrac{(q^{45},q^{63}, q^{99},  q^{117};q^{162})_\infty(q^{9}, q^{27}, q^{81}, q^{135}, q^{153};q^{162})_\infty^2}{(q^{18}, q^{144};q^{162})_\infty( q^{36},q^{54}, q^{72}, q^{90}, q^{108},q^{126};q^{162})_\infty^2}\\
		\nonumber &+q^{35}\dfrac{(q^{-18}, q^{27}, q^{45},q^{135}, q^{117},  q^{180};q^{162})_\infty(q^{9}, q^{63}, q^{81}, q^{99}, q^{153};q^{162})_\infty^2}{(q^{18}, q^{36},q^{54}, q^{72}, q^{90}, q^{108}, q^{126}, q^{144};q^{162})_\infty^2}\\
		&+q^{40}\dfrac{(q^{-36}, q^{9}, q^{45},q^{117}, q^{153},  q^{198};q^{162})_\infty(q^{27}, q^{63}, q^{81}, q^{99}, q^{135};q^{162})_\infty^2}{(q^{18}, q^{36},q^{54}, q^{72}, q^{90}, q^{108}, q^{126}, q^{144};q^{162})_\infty^2}.
	\end{align}
	Now extracting the terms involving $q^{9n+8}$, we have Theorem $ \eqref{18}(i)$. Proofs of Theorem $\ref{18}(ii)$ and $(iii )$ is similar to Theorem $\ref{18}(i)$, so we omit the proof.
\end{proof}  
\begin{theorem}{\label{TT1}}
	If
	\begin{eqnarray*}
		B_1^*(q)=\frac{(q^6, q^{20};q^{26})_\infty}{(q^7, q^{19}; q^{26})_\infty}= \displaystyle \sum_{n=0}^{\infty}a_n q^n, \\
		\frac{1}{B_2^*(q)}= \frac{(q^8, q^{18} ;q^{26})_\infty}{(q^5, q^{21}; q^{26})_\infty}= \displaystyle \sum_{n=0}^{\infty}b_n q^n, \\
		B_3^*(q)=\frac{(q^4, q^{22};q^{26})_\infty}{(q^9, q^{17}; q^{26})_\infty}= \displaystyle \sum_{n=0}^{\infty}c_n q^n, \\
		\frac{1}{B_4^*(q)}= \frac{(q^{10}, q^{16} ;q^{26})_\infty}{(q^3, q^{23}; q^{26})_\infty}= \displaystyle \sum_{n=0}^{\infty}d_n q^n, \\
		B_5^*(q)=\frac{(q^2, q^{24};q^{26})_\infty}{(q^{11}, q^{15}; q^{26})_\infty}= \displaystyle \sum_{n=0}^{\infty}e_n q^n, \\
		\frac{1}{B_6^*(q)}= \frac{(q^{12}, q^{14} ;q^{26})_\infty}{(q, q^{25}; q^{26})_\infty}= \displaystyle \sum_{n=0}^{\infty}f_n q^n, 
	\end{eqnarray*}
	then, we have\\	 $\hspace{2cm}~~(i) ~~a_{13n+11}=0, ~~~(ii)~~b_{13n+11}=0, ~~~(iii)~~c_{13n+7}=0$,\\ 
	$\hspace{2cm}~~(iv) ~~d_{13n+7}=0, ~~~(v)~~e_{13n+12}=0, ~~~(vi)~~f_{13n+12}=0$.

\end{theorem}
\begin{proof}
	Setting $t=26, s=13, r=7, p=13$ in $\eqref{15}$, we obtain
	\begin{align}{\label{46}}
		\nonumber	&\frac{(q^{26}, q^{26}, q^{20}, q^{6};q^{26})_\infty}{(q^{13}, q^{13}, q^{7}, q^{19}; q^{26})_\infty}=\frac{(q^{338}, q^{338}, q^{104}, q^{234};q^{338})_\infty}{(q^{13}, q^{325}, q^{91}, q^{247}; q^{338})_\infty}+q^7\frac{(q^{338}, q^{338}, q^{130}, q^{208};q^{338})_\infty}{(q^{39}, q^{299}, q^{91}, q^{247}; q^{338})_\infty}\\
		\nonumber&+q^{14}\frac{(q^{338}, q^{338}, q^{156}, q^{182};q^{338})_\infty}{(q^{65}, q^{273}, q^{91}, q^{247}; q^{338})_\infty}+q^{21}\frac{(q^{338}, q^{338}, q^{182}, q^{156};q^{338})_\infty}{(q^{91}, q^{247}, q^{91}, q^{247}; q^{338})_\infty}+q^{28}\frac{(q^{338}, q^{338}, q^{208}, q^{130};q^{338})_\infty}{(q^{117}, q^{221}, q^{91}, q^{247}; q^{338})_\infty}
	\end{align}
	\begin{align}
		\nonumber&+q^{35}\frac{(q^{338}, q^{338}, q^{234}, q^{104};q^{338})_\infty}{(q^{143}, q^{195}, q^{91}, q^{247}; q^{338})_\infty}+q^{42}\frac{(q^{338}, q^{338}, q^{260}, q^{78};q^{338})_\infty}{(q^{169}, q^{169}, q^{91}, q^{247}; q^{338})_\infty}+q^{49}\frac{(q^{338}, q^{338}, q^{286}, q^{52};q^{338})_\infty}{(q^{195}, q^{143}, q^{91}, q^{247}; q^{338})_\infty}\\
		\nonumber&+q^{56}\frac{(q^{338}, q^{338}, q^{312}, q^{26};q^{338})_\infty}{(q^{221}, q^{117}, q^{91}, q^{247}; q^{338})_\infty}+q^{63}\frac{(q^{338}, q^{338}, q^{338}, q^{0};q^{338})_\infty}{(q^{247}, q^{91}, q^{91}, q^{247}; q^{338})_\infty}+q^{70}\frac{(q^{338}, q^{338}, q^{364}, q^{-26};q^{338})_\infty}{(q^{273}, q^{65}, q^{91}, q^{247}; q^{338})_\infty}\\
		&+q^{77}\frac{(q^{338}, q^{338}, q^{393}, q^{-52};q^{338})_\infty}{(q^{299}, q^{39}, q^{91}, q^{247}; q^{338})_\infty}+q^{84}\frac{(q^{338}, q^{338}, q^{419}, q^{-78};q^{338})_\infty}{(q^{325}, q^{13}, q^{91}, q^{247}; q^{338})_\infty}.
	\end{align}
	Multiplying both sides of $\eqref{46}$ by $(q^{13};q^{13})_\infty^2/(q^{26};q^{26})_\infty^2$ and using $\eqref{aaa}$ , then extracting the terms involving $q^{13n+11}$, we have Theorem $\ref{TT1}(i)$. Proofs of Theorem $\ref{TT1}(ii)$-$(vi)$ is similar to Theorem $\ref{TT1}(i)$, so we omit the proof.
\end{proof}
\begin{theorem}{\label{tt1}}
	\begin{eqnarray*}
		\frac{1}{C_1^*(q)}= \frac{(q^8, q^{22} ;q^{30})_\infty}{(q^7, q^{23}; q^{30})_\infty}= \displaystyle \sum_{n=0}^{\infty}a_n q^n, \\
		C_4^*(q)=\frac{(q^4, q^{26};q^{30})_\infty}{(q^{11}, q^{19}; q^{30})_\infty}= \displaystyle \sum_{n=0}^{\infty}b_n q^n, \\
		C_6^*(q)=\frac{(q^2, q^{28};q^{30})_\infty}{(q^{13}, q^{17}; q^{30})_\infty}= \displaystyle \sum_{n=0}^{\infty}c_n q^n, \\
		\frac{1}{C_7^*(q)}= \frac{(q^{14}, q^{16} ;q^{30})_\infty}{(q, q^{29}; q^{30})_\infty}= \displaystyle \sum_{n=0}^{\infty}d_n q^n, \\
	\end{eqnarray*}
	then,	we have\\	 $\hspace{2cm}~~(i) ~~a_{15n+2}=0, ~~~(ii)~~b_{15n+9}=0,$\\ 
	$\hspace{2cm}~~(iv) ~~c_{15n+4}=0, ~~~(v)~~d_{15n+14}=0, ~~~(vi)~~f_{13n+12}=0$.
\end{theorem}
\begin{proof}
	Setting $t=30, s=15, r=7, p=15$ in $\eqref{15}$, we obtain
	\begin{align}	{\label{r7}}
		\nonumber	&\frac{(q^{30}, q^{30}, q^{22}, q^{8};q^{30})_\infty}{(q^{15}, q^{15}, q^{7}, q^{23}; q^{30})_\infty}=\frac{(q^{450}, q^{450}, q^{120}, q^{330};q^{450})_\infty}{(q^{15}, q^{435}, q^{105}, q^{345}; q^{450})_\infty}+q^{7}\frac{(q^{450}, q^{450}, q^{150}, q^{300};q^{450})_\infty}{(q^{45}, q^{405}, q^{105}, q^{345}; q^{450})_\infty}\\
		\nonumber&+q^{14}\frac{(q^{450}, q^{450}, q^{180}, q^{270};q^{450})_\infty}{(q^{75}, q^{375}, q^{105}, q^{345}; q^{450})_\infty}+q^{21}\frac{(q^{450}, q^{450}, q^{210}, q^{240};q^{450})_\infty}{(q^{105}, q^{345}, q^{105}, q^{345}; q^{450})_\infty}+q^{28}\frac{(q^{450}, q^{450}, q^{240}, q^{210};q^{450})_\infty}{(q^{135}, q^{315}, q^{105}, q^{345}; q^{450})_\infty}\\
		\nonumber&+q^{35}\frac{(q^{450}, q^{450}, q^{270}, q^{180};q^{450})_\infty}{(q^{165}, q^{285}, q^{105}, q^{345}; q^{450})_\infty}+q^{42}\frac{(q^{450}, q^{450}, q^{300}, q^{150};q^{450})_\infty}{(q^{195}, q^{255}, q^{105}, q^{345}; q^{450})_\infty}+q^{49}\frac{(q^{450}, q^{450}, q^{330}, q^{120};q^{450})_\infty}{(q^{225}, q^{225}, q^{105}, q^{345}; q^{450})_\infty}\\	
		\nonumber&+q^{56}\frac{(q^{450}, q^{450}, q^{360}, q^{90};q^{450})_\infty}{(q^{255}, q^{195}, q^{105}, q^{345}; q^{450})_\infty}
		+q^{63}\frac{(q^{450}, q^{450}, q^{390}, q^{60};q^{450})_\infty}{(q^{285}, q^{165}, q^{105}, q^{345}; q^{450})_\infty}+q^{70}\frac{(q^{450}, q^{450}, q^{420}, q^{30};q^{450})_\infty}{(q^{315}, q^{135}, q^{105}, q^{345}; q^{450})_\infty}\\
		\nonumber&+q^{77}\frac{(q^{450}, q^{450}, q^{450}, q^{0};q^{450})_\infty}{(q^{345}, q^{105}, q^{105}, q^{345}; q^{450})_\infty}+q^{84}\frac{(q^{450}, q^{450}, q^{480}, q^{-30};q^{450})_\infty}{(q^{375}, q^{75}, q^{105}, q^{345}; q^{450})_\infty}+q^{91}\frac{(q^{450}, q^{450}, q^{510}, q^{-60};q^{450})_\infty}{(q^{405}, q^{45}, q^{105}, q^{345}; q^{450})_\infty}\\&+q^{98}\frac{(q^{450}, q^{450}, q^{540}, q^{-90};q^{450})_\infty}{(q^{435}, q^{15}, q^{105}, q^{345}; q^{450})_\infty}.
	\end{align}
	Multiplying both sides of $\eqref{r7}$ by $(q^{15};q^{15})_\infty^2/(q^{30};q^{30})_\infty^2$ and using $\eqref{aaa}$ , then extracting the terms involving $q^{15n+2}$, we have Theorem $\ref{tt1}(i)$. Proofs of Theorem $\ref{tt1}(ii)$-$(iv)$ is similar to Theorem $\ref{tt1}(i)$, so we omit the proof.
\end{proof}
\section{Theta function identities for the continued fractions}
In this section, we prove some algebraic relations for $A_i, B_j$ and $C_l$ where $1 \le i \le 4$, $1\le j \le 6$ and $1 \le l \le 7  $ of order ten, eighteen and twenty six respectively. Also we deduce color partition for some theta function identities. 
\begin{theorem}{\label{T2}}
	We have
	\begin{enumerate}[i.]
		\item $\dfrac{1}{A_1(q)} \mp q^{1/2}A_1(q)= \dfrac{f(\mp q^{1/2}, \mp q^{17/2}) \varphi(\pm q^{9/2})}{f(-q^4, -q^5) \psi(q^9)}$, \label{03}\\
		\item $\dfrac{1}{A_2(q)} \mp q^{3/2}A_2(q)= \dfrac{f(\mp q^{3/2}, \mp q^{15/2}) \varphi(\pm q^{9/2})}{f(-q^3, -q^6) \psi(q^9)}$, \label{04}\\
		\item $\dfrac{1}{A_3(q)} \mp q^{5/2}A_3(q)= \dfrac{f(\mp q^{5/2}, \mp q^{13/2}) \varphi(\pm q^{9/2})}{f(-q^2, -q^7) \psi(q^9)}$, \label{05}\\
		\item $\dfrac{1}{A_4(q)} \mp q^{7/2}A_4(q)= \dfrac{f(\mp q^{7/2}, \mp q^{11/2}) \varphi(\pm q^{9/2})}{f(-q, -q^8) \psi(q^9)}$, \label{06}\\
		\item $\displaystyle \prod_{i=1,3,4}^{}\left(\dfrac{1}{A_i(q)}+q^{(2(i-1)+1)/2}A_i(q)\right)=\dfrac{\varphi^3(-q^{9/2})(-q^{\pm(2(i-1)+1))/2};q^9)_\infty}{\psi^3(q^9)(q^{1\pm};q^3)_\infty}$. \label{07}
	\end{enumerate}
\end{theorem}
\begin{proof}
	From $\eqref{4}$, we obtain
	\begin{equation}{\label{20}}
		\frac{1}{\sqrt{q^{1/4}A_1(q)}}-\sqrt{q^{1/4}A_1(q)}=\frac{f(-q^5, -q^{13})-q^{1/4}f(-q^4, -q^{14})}{\sqrt{q^{1/4}f(-q^4, -q^{14})f(-q^5, -q^{13})}}.
	\end{equation}
	From \cite[ p. 46, Entry (i) and (ii)]{Bernd}, we have
	\begin{equation}\label{20A}
		f(a, b)=f(a^3b, ab^3)+ a f(b/a, a^5b^3)
	\end{equation}
	Setting $\left\{a=-q^{1/4}, b=q^{17/4}\right\}$ and $\left\{a=q^{1/4}, b=-q^{17/4}\right\}$ in $\eqref{20A}$, we have
	\begin{equation}{\label{21}}
		f(-q^{1/4}, q^{17/4})=f(-q^5, -q^{13})-q^{1/4}f(-q^4, -q^{14}) 
	\end{equation}
	and
	\begin{equation}{\label{22}	}
		f(q^{1/4}, -q^{17/4})=f(-q^5, -q^{13})-q^{1/4}f(-q^4, -q^{14}). 
	\end{equation}
	Employing $\eqref{21}$ in $\eqref{20}$, we obtain
	\begin{equation}{\label{23}}
		\frac{1}{\sqrt{q^{1/4}A_1(q)}}-\sqrt{q^{1/4}A_1(q)}=\frac{	f(-q^{1/4}, q^{17/4})}{\sqrt{q^{1/4}f(-q^4, -q^{14})f(-q^5, -q^{13})}}.
	\end{equation} 
	Similarly, from $\eqref{4}$ and applying $\eqref{22}$, we obtain
	\begin{equation}{\label{24}}
		\frac{1}{\sqrt{q^{1/4}A_1(q)}} +\sqrt{q^{1/4}A_1(q)}=\frac{	f(q^{1/4}, -q^{17/4})}{\sqrt{q^{1/4}f(-q^4, -q^{14})f(-q^5, -q^{13})}}.
	\end{equation}
	Combining $\eqref{23}$ and $\eqref{24}$, we arrive at
	\begin{equation}{\label{25}}
		\frac{1}{q^{1/4}A_1(q)}-q^{1/4}A_1(q)=\frac{f(-q^{1/4}, q^{17/4})f(q^{1/4}, -q^{17/4})}{q^{1/4}f(-q^4, -q^{14})f(-q^5, -q^{13})}.
	\end{equation}
	From \cite[p. 46, Entry 30 (i) and (iv)]{Bernd}, we note that
	\begin{equation}{\label{26}}
		f(a, ab^2) f(b, a^2b)=f(a, b) \psi(ab),
	\end{equation}
	\begin{equation}{\label{27}}
		f(a, b)f(-a, -b)=f(-a^2, -b^2 ) \varphi(-ab).
	\end{equation}
	Setting $\left\{a=-q^{4}, b=-q^{5}\right\}$ in $\eqref{26}$ and $\left\{a=-q^{1/4}, b=q^{17/4}\right\}$ in $\eqref{27}$, we obtain
	\begin{equation}{\label{28}}
		f(-q^4, -q^{14})f(-q^5, -q^{13})=f(-q^4, -q^5)\psi(q^9),
	\end{equation}
	\begin{equation}\label{29}
		f(-q^{1/4}, q^{17/4})f(q^{1/4}, -q^{17/4})=f(-q^{1/2}, q^{17/2}) \varphi(q^{9/2})
	\end{equation}
	respectively. Employing $\eqref{28}$ and $\eqref{29}$ in $\eqref{25}$, we have
	\begin{equation*}
		\frac{1}{A_1(q)} - q^{1/2}A_1(q)= \frac{f(- q^{1/2}, -q^{917/2}) \varphi(q^{9/2})}{f(-q^4, -q^5) \psi(q^9)}.
	\end{equation*}
	Squaring $\eqref{24}$, we obtain
	\begin{equation}{\label{30}}
		\frac{1}{q^{1/4}A_1(q)} +q^{1/4}A_1(q)=\frac{f^2(q^{1/4}, -q^{17/4})}{q^{1/4}f(-q^4, -q^{14})f(-q^5, -q^{13})}-2.
	\end{equation}
	From \cite[p. 46, Entry 30 (v), (vi)]{Bernd}, we note that
	\begin{equation}{\label{31}}
		f^2(a, b)= f(a^2, b^2) \varphi(ab)+ 2a f(b/a, a^3b)\psi (a^2b^2).
	\end{equation}
	Setting  $\left\{a=q^{1/4}, b=-q^{17/4}\right\}$ in $\eqref{31}$, we obtain
	\begin{equation}{\label{32}}
		f^2(q^{1/4}, -q^{17/4})=f(q^{1/2}, q^{17/2})\varphi(-q^{9/2})+2 q^{1/4}f(-q^4, -q^{14}) \psi (q^9).
	\end{equation}
	Employing $\eqref{32}$ and $\eqref{28}$ in $\eqref{30}$, we have
	\begin{equation*}
		\frac{1}{A_1(q)} + q^{1/2}A_1(q)= \frac{f(q^{1/2}, q^{17/2}) \varphi(-q^{9/2})}{f(-q^4, -q^5) \psi(q^9)}.
	\end{equation*}
	Proofs of $\eqref{04}$-$\eqref{06}$ are identical to the proof of $\eqref{03}$, so we omit the proof.\\
	From $\eqref{03}$, $\eqref{04}$ and $\eqref{06}$, we obtain
	\begin{align}{\label{37}}
		\displaystyle \prod_{i=1,3,4}^{}\left(\dfrac{1}{A_i(q)}+q^{(2(i-1)+1)/2}A_i(q)\right)=\frac{\varphi^3(-q^{9/2})f(q^{1/2}, q^{17/2})f(q^{5/2}, q^{13/2})f(q^{7/2}, q^{11/2}) }{\psi^3(q^9)f(-q, -q^8)f(-q^2, -q^7)f(-q^4, -q^5)}.
	\end{align}
	From \cite[p. 349, Entry 2]{Bernd}, we note that
	\begin{equation}{\label{38}}
		f(-q, -q^8)f(-q^2, -q^7)f(-q^4, -q^5)=\frac{f(-q)f^3(-q^9)}{f(-q^3)}.
	\end{equation}
	Employing $\eqref{38}$ in $\eqref{37}$, we arrive at $\eqref{07}$.
\end{proof}
\begin{theorem}{\label{T3}}
	We have
	\begin{enumerate}[i.]
		\item $\dfrac{1}{B_1(q)} \mp q^{1/2}B_1(q)= \dfrac{f(\mp q^{1/2}, \mp q^{25/2}) \varphi(\pm q^{13/2})}{f(-q^6, -q^7) \psi(q^{13})}$, \label{08}\\
		\item $\dfrac{1}{B_2(q)} \mp q^{3/2}B_2(q)= \dfrac{f(\mp q^{3/2}, \mp q^{23/2}) \varphi(\pm q^{13/2})}{f(-q^5, -q^8) \psi(q^{13})}$, \label{09}\\
		\item $\dfrac{1}{B_3(q)} \mp q^{5/2}B_3(q)= \dfrac{f(\mp q^{5/2}, \mp q^{21/2}) \varphi(\pm q^{13/2})}{f(-q^4, -q^8) \psi(q^{13})}$, \label{010}\\
		\item $\dfrac{1}{B_4(q)} \mp q^{7/2}B_4(q)= \dfrac{f(\mp q^{7/2}, \mp q^{19/2}) \varphi(\pm q^{13/2})}{f(-q^3, -q^{10}) \psi(q^{13})}$, \label{011}\\
		\item $\dfrac{1}{B_5(q)} \mp q^{9/2}B_5(q)= \dfrac{f(\mp q^{9/2}, \mp q^{17/2}) \varphi(\pm q^{13/2})}{f(-q^2, -q^{11}) \psi(q^{13})}$, \label{012}\\
		\item $\dfrac{1}{B_6(q)} \mp q^{11/2}B_6(q)= \dfrac{f(\mp q^{11/2}, \mp q^{15/2}) \varphi(\pm q^{13/2})}{f(-q, -q^{12}) \psi(q^{13})}$, \label{013}\\
		\item $\displaystyle \prod_{i=1}^{6}\left(\dfrac{1}{B_i(q)}+q^{(2(i-1)+1)/2}B_i(q)\right)=\dfrac{\varphi^6(-q^{13/2})(-q^{\pm(2(i-1)+1))/2 };q^{13})_\infty}{\psi^6(q^{13})f(-q)}$. \label{014}
	\end{enumerate}
	
\end{theorem}
\begin{proof}
	Proofs of $\eqref{08}$-$\eqref{013}$ are identical to the proof of Theorem $\ref{T2}$$\eqref{03}$, so we omit the proof.\\
	From  $\eqref{08}$-$\eqref{013}$, we have
	\begin{equation*}
		\displaystyle \prod_{i=1}^{6}\left(\frac{1}{B_i(q)}+q^{(2(i-1)+1)/2}B_i(q)\right)
	\end{equation*}
	\begin{equation}{\label{39}}
		=\frac{\varphi^6(-q^{13/2}f(q^{1/2}, q^{25/2})f(q^{3/2}, q^{23/2})f(q^{5/2}, q^{21/2})f(q^{7/2}, q^{19/2})f(q^{9/2}, q^{17/2})f(q^{11/2}, q^{15/2}))}{\psi^6(q^{13})f(-q, -q^{12})f(-q^2, -q^{11})f(-q^3, -q^{10})f(-q^4, -q^{9})f(-q^5, -q^{8})f(-q^6, -q^{7})}.
	\end{equation}
	From \cite[p. 376, Entry 8 (ii)]{Bernd}, we note that
	\begin{equation}{\label{40}}
		f(-q, -q^{12})f(-q^2, -q^{11})f(-q^3, -q^{10})f(-q^4, -q^{9})f(-q^5, -q^{8})f(-q^6, -q^{7})=f(-q)f^5(-q^{13}).
	\end{equation}
	Employing $\eqref{40}$ in $\eqref{39}$, we arrive at $\eqref{014}$.
\end{proof}
\begin{theorem}{\label{t4}}
	We have
	\begin{enumerate}[i.]
		\item $\dfrac{1}{C_1(q)} \mp q^{1/2}C_1(q)= \dfrac{f(\mp q^{1/2}, \mp q^{29/2}) \varphi(\pm q^{15/2})}{f(-q^7, -q^8) \psi(q^{15})}$, \label{t41}\\
		\item $\dfrac{1}{C_2(q)} \mp q^{3/2}C_2(q)= \dfrac{f(\mp q^{3/2}, \mp q^{27/2}) \varphi(\pm q^{15/2})}{f(-q^6, -q^9) \psi(q^{15})}$, \label{t42}\\
		\item $\dfrac{1}{C_3(q)} \mp q^{5/2}C_3(q)= \dfrac{f(\mp q^{5/2}, \mp q^{25/2}) \varphi(\pm q^{15/2})}{f(-q^5, -q^{10}) \psi(q^{15})}$, \label{t43}\\
		\item $\dfrac{1}{C_4(q)} \mp q^{7/2}C_4(q)= \dfrac{f(\mp q^{7/2}, \mp q^{23/2}) \varphi(\pm q^{15/2})}{f(-q^4, -q^{11}) \psi(q^{15})}$, \label{t44}\\
		\item $\dfrac{1}{C_5(q)} \mp q^{9/2}C_5(q)= \dfrac{f(\mp q^{9/2}, \mp q^{21/2}) \varphi(\pm q^{15/2})}{f(-q^3, -q^{12}) \psi(q^{15})}$, \label{t45}\\
		\item $\dfrac{1}{C_6(q)} \mp q^{11/2}C_6(q)= \dfrac{f(\mp q^{11/2}, \mp q^{19/2}) \varphi(\pm q^{15/2})}{f(-q^2, -q^{13}) \psi(q^{15})}$, \label{t46}\\
		\item $\dfrac{1}{C_7(q)} \mp q^{13/2}C_7(q)= \dfrac{f(\mp q^{13/2}, \mp q^{17/2}) \varphi(\pm q^{15/2})}{f(-q, -q^{14}) \psi(q^{15})}$, \label{t47}\\
	\end{enumerate}
\end{theorem}
\begin{proof}
	Proofs of $\eqref{t41}$-$\eqref{t47}$ are identical to the proof of Theorem $\ref{T2}$$\eqref{03}$, so we omit the proof.
\end{proof}

\begin{theorem}{\label{T5}}
	Let $X_1(n)$  denote the number of partitions of $n$ into parts congruent to $\pm 3, \pm 6, \pm 15 $ or $\pm 18\pmod {36}$ such that the parts congruent to $\pm 6$ and $\pm 18 \pmod {36}$ have two colors. Let $X_2(n)$  denote the number of partitions of $n$ into parts congruent to $\pm 3, \pm 12, \pm 15 $ or $\pm 18\pmod {36}$ such that the parts congruent to $\pm 12$ and $\pm 18 \pmod {36}$ have two colors. Let $X_3(n)$  denote the number of partitions of $n$ into parts congruent to $\pm 6, \pm 9$ and $\pm 12\pmod {36}$ with two colors. Then for any $n \ge 3$,
	\begin{equation*}
		X_1(n)-X_2(n-3)-X_3(n)=0.
	\end{equation*}
\end{theorem}
\begin{proof}
	Applying $\eqref{5}$ , $\eqref{49}$,  $\eqref{50}$ and changing $q$ by $q^2$ in Theorem $\ref{T2}$ $\eqref{04}$  we deduce
	\begin{equation}{\label{41}}
		\frac{( q^{\pm 12}; q^{36})_\infty}{(q^{\pm 6}; q^{36})_\infty}-q^3\frac{(q^{\pm 6}; q^{36})_\infty}{( q^{\pm 12}; q^{36})_\infty}-\frac{(q^{\pm 3, \pm 15};q^{36})_\infty(q^{\pm 18}; q^{36})_\infty^2}{(q^{\pm 6, \pm 12};q^{36})_\infty(q^{\pm 9}; q^{36})_\infty^2}=0.
	\end{equation}
	Dividing $\eqref{41}$ by $(q^{\pm 3, \pm 6, \pm 12, \pm 15};q^{36})_\infty(q^{\pm 18}; q^{36})_\infty^2$, we have
	\begin{equation}{\label{43}}
		\frac{1}{(q^{\pm 6, \pm 18}; q^{36})_\infty^2 (q^{\pm 3 \pm 15}; q^{36})_\infty}-\frac{q^3}{(q^{\pm 12, \pm 18}; q^{36})_\infty^2 (q^{\pm 3, \pm 15}; q^{36})_\infty}-\frac{1}{(q^{\pm 6, \pm 9, \pm 12}; q^{36})_\infty^2 }=0.
	\end{equation}
	The above quotients of $\eqref{43}$ represents generating functions for $X_1(n)$, $X_2(n)$ and $X_3(n)$ respectively. Hence  $\eqref{43}$ is equivalent to
	\begin{align*}
		\sum_{n=0}^{\infty}X_1(n)q^n-q^3\sum_{n=0}^{\infty}X_2(n)q^n-\sum_{n=0}^{\infty}X_3(n)q^n&=0.
	\end{align*}
	Equating coefficient of $q^n$ on both sides, we arrive at the desired result.
\end{proof}
\noindent For $n=9$, Theorem $\ref{T5}$.  is verified.
\begin{center}
	\begin{tabular}{ll}
		\hline
		$X_1(9)=3:$ &$6_r+3, 6_g+3, 3+3+3$  \\
		\hline
		$X_2(6)=1:$ &$3+3$  \\
		\hline
		$X_3(9)=2:$ &$9_r, 9_g$  \\
		\hline
	\end{tabular}	
	\label{table 2}
\end{center}
\begin{theorem}{\label{T6}}
	Let $Y_1(n)$  denote the number of partitions of $n$ into parts congruent to $\pm 1, \pm 12, \pm 25 $ or $\pm 26 \pmod {52}$ such that the parts congruent to $\pm 12$ and $\pm 26 \pmod {52}$ have two colors. Let $Y_2(n)$  denote the number of partitions of $n$ into parts congruent to $\pm 1, \pm 14, \pm 25 $ or $\pm 26\pmod {52}$ such that the parts congruent to $\pm 14$ and $\pm 26 \pmod {52}$ have two colors. Let $Y_3(n)$  denote the number of partitions of $n$ into parts congruent to $\pm 12, \pm 13$ and $\pm 14\pmod {52}$ with two colors. Then for any $n \ge 1$,
	\begin{equation*}
		Y_1(n)-Y_2(n-1)-Y_3(n)=0.
	\end{equation*}
\end{theorem}
\begin{proof}
	Applying $\eqref{8}$ , $\eqref{49}$,  $\eqref{50}$ and changing $q$ by $q^2$ in Theorem $\ref{T3}$ $\eqref{08}$ and then dividing by $(q^{\pm 1, \pm 12, \pm 14, \pm 25};q^{52})_\infty(q^{\pm 26}; q^{52})_\infty^2$, we obtain
	\begin{equation}{\label{44}}
		\frac{1}{(q^{\pm 12, \pm 26}; q^{52})_\infty^2 (q^{\pm 1, \pm 25 }; q^{52})_\infty}-\frac{q}{(q^{\pm 14, \pm 26 }; q^{52})_\infty^2 (q^{\pm 1, \pm 25}; q^{52})_\infty}-\frac{1}{(q^{\pm 12, \pm 14, \pm 13 }; q^{52})_\infty^2 }=0.
	\end{equation}
	The above quotients of $\eqref{44}$ represents generating functions for $Y_1(n), Y_2(n)$ and $Y_3(n)$ respectively. Hence  $\eqref{44}$ is equivalent to
	\begin{align*}
		\sum_{n=0}^{\infty}Y_1(n)q^n-q\sum_{n=0}^{\infty}Y_2(n)q^n-\sum_{n=0}^{\infty}Y_3(n)q^n&=0.
	\end{align*}
	Equating coefficient of $q^n$ on both sides, we arrive at the desired result.
\end{proof}
\noindent \noindent For $n=12$, Theorem $\ref{T6}$.  is verified.
\begin{center}
	\begin{tabular}{ll}
		\hline
		$Y_1(12)=3:$ &$12_r, 12_g, 1+1+1+1+1+1+1+1+1+1+1+1$  \\
		\hline
		$Y_2(11)=1:$ &$1+1+1+1+1+1+1+1+1+1+1$  \\
		\hline
		$Y_3(12)=2:$ &$12_r, 12_g$  \\
		\hline
	\end{tabular}	
	\label{table 3}
\end{center}
\begin{theorem}{\label{T7}}
	Let $Z_1(n)$  denote the number of partitions of $n$ into parts congruent to $\pm 1, \pm 14, \pm 29 $ or $\pm 30 \pmod {60}$ such that the parts congruent to $\pm 14$ and $\pm 30 \pmod {60}$ have two colors. Let $Z_2(n)$  denote the number of partitions of $n$ into parts congruent to $\pm 1, \pm 16, \pm 29 $ or $\pm 30\pmod {60}$ such that the parts congruent to $\pm 16$ and $\pm 30 \pmod {60}$ have two colors. Let $Z_3(n)$  denote the number of partitions of $n$ into parts congruent to $\pm 14, \pm 15$ and $\pm 16\pmod {}$ with two colors. Then for any $n \ge 1$,
	\begin{equation*}
		Z_1(n)-Z_2(n-1)-Z_3(n)=0.
	\end{equation*}
\end{theorem}
\begin{proof}
	Applying $\eqref{r1}$ , $\eqref{49}$,  $\eqref{50}$ and changing $q$ by $q^2$ in Theorem $\ref{t4}$ $\eqref{t41}$ and then dividing by $(q^{\pm 1, \pm 14, \pm 16, \pm 29};q^{60})_\infty(q^{\pm 30}; q^{60})_\infty^2$, we obtain
	\begin{equation}{\label{t48}}
		\frac{1}{(q^{\pm 14, \pm 30}; q^{60})_\infty^2 (q^{\pm 1, \pm 29 }; q^{60})_\infty}-\frac{q}{(q^{\pm 16, \pm 30 }; q^{60})_\infty^2 (q^{\pm 1, \pm 29}; q^{60})_\infty}-\frac{1}{(q^{\pm 14, \pm 15, \pm 16 }; q^{60})_\infty^2 }=0.
	\end{equation}
	The above quotients of $\eqref{t48}$ represents generating functions for $Z_1(n), Z_2(n)$ and $Z_3(n)$ respectively. Hence  $\eqref{t48}$ is equivalent to
	\begin{align*}
		\sum_{n=0}^{\infty}Z_1(n)q^n-q\sum_{n=0}^{\infty}Z_2(n)q^n-\sum_{n=0}^{\infty}Z_3(n)q^n&=0.
	\end{align*}
	Equating coefficient of $q^n$ on both sides, we arrive at the desired result.
\end{proof}
\noindent \noindent For $n=16$, Theorem $\ref{T7}$.  is verified.
\begin{center}
	\begin{tabular}{ll}
		\hline
		$Z_1(12)=3:$ &$14_r+1+1, 14_g+1+1, 1+1+1+1+1+1+1+1+1+1+1+1+1+1+1+1$  \\
		\hline
		$Z_2(11)=1:$ &$1+1+1+1+1+1+1+1+1+1+1+1+1+1+1$  \\
		\hline
		$Z_3(12)=2:$ &$16_r, 16_g$  \\
		\hline
	\end{tabular}	
	\label{table 3}
\end{center}


\begin{thebibliography}{99}
		\bibitem {Bernd}
	Berndt, B.C.: Ramanujan's Notebooks, Part III, 
	Springer, New York (1991)
	\bibitem {1998}
	Berndt, B.C.: Ramanujan's Notebooks, Part V, 
	Springer, New York (1998)
	\bibitem {1979}G. E. Andrews and D. Bressoud: Vanishing coefficients in infinite product expansion, J. Aust. Math. Soc. Ser., \textbf{27}  199-202 (1979)
\end{thebibliography}
\end{document}